\documentclass[12pt,reqno,a4paper]{amsart}
\usepackage{extsizes}
\usepackage{hyperref}
\usepackage{amsmath,amsfonts,enumerate,amssymb}
\usepackage{mathtools}

\newcommand{\C}{\mathbb{C}}

\newcommand{\js}{\mathcal Z}
\newcommand{\proj}{\mathcal P}
\newcommand{\s}{\text{sp}}
\newcommand{\Cu}{\text{Cu}}
\newcommand{\id}{\text{id}}
\newcommand{\Ad}{\text{Ad}}
\newcommand{\coz}{\text{Coz}}
\newcommand{\qt}{\text{QT}}
\newcommand{\K}{\mathcal{K}}

\newtheorem{innercustomthm}{Theorem}
\newenvironment{customthm}[1]
  {\renewcommand\theinnercustomthm{#1}\innercustomthm}
  {\endinnercustomthm}

\newtheorem{thm}{Theorem}[section]
\newtheorem{prop}[thm]{Proposition}
\newtheorem{lem}[thm]{Lemma}
\newtheorem{cor}[thm]{Corollary}

\newtheorem{remark}[thm]{Remark}

\title{Contractible Cuntz Classes in \linebreak $\mathcal{Z}$-stable C$^*$-algebras}
\author{Chrisil Ouseph}
\address{Department of Mathematics, Purdue University, 150 N University St, West Lafayette IN 47907, USA}
\email{couseph@purdue.edu}
\author{Andrew S. Toms}
\address{Department of Mathematics, Purdue University, 150 N University St, West Lafayette IN 47907, USA}
\email{atoms@purdue.edu}

\begin{document}
\begin{abstract}
Let $A$ be a unital, simple and $\js$-stable C$^*$-algebra.  We show that the set of positive elements in $A$ (resp. $A \otimes \mathcal{K}$) belonging to a fixed non-compact Cuntz class is contractible as a topological subspace of $A$ (resp. $A \otimes \mathcal{K}$).  In light of earlier work by Zhang, Jiang and Hua in the compact case, we deduce a complete calculation of the homotopy groups of Cuntz classes for these algebras.
\end{abstract}
\maketitle
\section{Introduction}
A dimension function on a C$^*$-algebra abstracts of the fundamental notion of rank for operators on a finite-dimensional complex vector space.  In his foundational 1978 paper, Cuntz introduced an invariant for C$^*$-algebras, now known simply as the Cuntz semigroup, whose states are precisely the dimension functions on the algebra (\cite{Cuntz1978}). Its construction is based on a comparison relation for positive elements which saw significant applications in the work of Kirchberg and R\o rdam on purely infinite C$^*$-algebras in the 1990s (\cite{Kirchberg2002InfiniteNC}).  Applications of the Cuntz semigroup proper, however, remained sparse until the second author's discovery in 2005 that it could distinguish simple nuclear separable C$^*$-algebras which agreed on K-theory and traces, invariants that had heretofore sufficed to classify these algebras up to isomorphism (\cite{TomsAnnals}).  This catalyzed a surge of research into the structure of the Cuntz semigroup and its bearing on the structure and classification of nuclear C$^*$-algebras, the latter driven by the formulation of the now largely proved Toms-Winter conjecture (\cite{WinterICM}).\\

The Cuntz semigroup consists of equivalence classes of positive elements in the stabilization of a C$^*$-algebra which, roughly, generalizes Murray-von Neumann equivalence of projections. For simple stably finite C$^*$-algebras the two notions agree on projections, and the Cuntz classes of projections are referred to as being compact.  In the early 1990s Zhang initiated the study of the homotopy groups of Murray-von Neumann equivalence classes as topological subspaces of a simple C$^*$-algebra $A$ of real rank zero, and found that these groups alternate between $K_0(A)$ and $K_1(A)$ (\cite{Zhang}). This manifestation of Bott periodicity was subsequently generalized to the setting of unital simple $\mathcal{Z}$-stable algebras in an unpublished note of Jiang, the details of which have recently been put on solid footing by Hua (\cite{Jiang}, \cite{Hua}). These results can be viewed equivalently as a calculation of the homotopy groups of the compact Cuntz classes of these algebras (see Proposition \ref{proj_def_ret}), which begs the question of what these groups might be for non-compact classes.\\

In \cite{path_connected} and \cite{trivial_hom} the second author initiated the study of the homotopy groups of non-compact Cuntz classes, proving ultimately that relative to the original C$^*$-algebra (as opposed to its stabilization) they vanish in all dimensions for a simple unital separable exact and $\mathcal{Z}$-stable C$^*$-algebra of real rank zero.  The method of proof relied on the comparison theory of positive elements, the approximation of non-compact Cuntz classes as suprema of sequences of compact classes, and a delicately constructed asymptotic unitary equivalence.  Here by new methods we strengthen this result substantially in several directions to what we surmise is its most general form in the simple unital case.  Let $A$ be a C$^*$-algebra and let $a \in A \otimes \mathcal{K}$ be positive.  Set
\[
\langle a \rangle = \{ b \in (A \otimes \mathcal{K})_+ \ | \ b \sim a \},
\]
where $\sim$ denotes the Cuntz equivalence relation.  If $a \in A_+$, set
\[
S_a = \langle a \rangle \cap A,
\]
where $A$ is embedded in $A \otimes \mathcal{K}$ as $A \otimes p$ for a minimal projection $p \in \mathcal{K}$.
The main result of this paper is the following:

\begin{customthm}{I}
\label{main_result}
    Consider a unital, simple, and $\js$-stable $C^*$-algebra $A$. If \mbox{$a\in A_+$} with non-compact Cuntz class, then $S_a$ is a contractible topological subspace of $A$.  If instead $a \in (A\otimes\K)_+$ with non-compact Cuntz class, then $\langle a \rangle$ is a contractible topological subspace of $A\otimes \K$.
\end{customthm}

Thus, we have dispensed with the assumptions of separability, exactness, and real rank zero, strengthened the conclusion from the vanishing of homotopy groups to contractibility, and established this contractibility for the set of positive elements belonging to a fixed non-compact Cuntz class relative to both the original C$^*$-algebra and to its stabilization.  It seems unlikely that the assumption of $\mathcal{Z}$-stability can be dropped, but proving as much will probably require a new application of Villadsen's methods, more likely those of \cite{Villadsen2} than \cite{Villadsen1999OnTS}.\\

Theorem \ref{main_result} combined with separate works of Zhang, Jiang and Hua (\cite{Zhang}, \cite{Jiang} and \cite{Hua}, respectively) gives a complete calculation of the homotopy groups of Cuntz classes for unital simple $\mathcal{Z}$-stable C$^*$-algebras:

\begin{cor}\label{homgroups}
Let $A$ be a unital simple $\mathcal{Z}$-stable $C^*$-algebra, and let $a \in A$ be positive.  If $\langle a \rangle$ is compact in the Cuntz semigroup $\mathrm{Cu}(A)$, then for $k \in \mathbb{N}$,
\[
\pi_k(S_a) = \left\{ \begin{array}{ll} K_0(A), & k \ \mathrm{even} \\ K_1(A), & k \ \mathrm{odd} \end{array} \right.
\]
If, alternatively, $\langle a \rangle$ is non-compact, then $\pi_k(S_a) = 0$ for each $k \in \mathbb{N}$.
\end{cor}

We close this section with a note on our deeper motivations.  The Toms-Winter conjecture predicts the equivalence of three regularity properties for a simple separable nuclear unital non-elementary C$^*$-algebra $A$:  finite nuclear dimension, $\mathcal{Z}$-stability, and strict comparison.  Finite nuclear dimension and $\mathcal{Z}$-stability are by now known to be equivalent and together imply strict comparison (\cite{rordam2004stablerealrankzabsorbing}, \cite{Winterpure}, \cite{Zstabledimnuc}).   The open implication of the conjecture is only known to hold under additional hypotheses such as stable rank one (\cite{Thielstablerankone}) or a compact extreme tracial boundary of finite covering dimension (\cite{kirchbergrordamfindim}, \cite{tomsetalfindim}, \cite{Satofindim}) .  One difficulty in resolving this problem is that we have yet to find a general method for passing from strict comparison to the statement that every possible strictly positive lower semicontinuous and affine function on the tracial state space occurs as the rank function of a positive operator $a \in A \otimes \mathcal{K}$, where the rank function is defined by
\[
\tau \mapsto d_\tau(a) := \lim_{n \to \infty} \tau(a^{1/n}), \ \tau \in \mathrm{T}(A).
\]

What follows is informed speculation at best, but let us suppose, for the sake of exposition, that the extreme tracial boundary $\partial_e \mathrm{T}(A)$ is compact.  Suppose further that for each $\tau \in \partial_e \mathrm{T}(A)$ and real numbers $0 <r<s$, there is an open neighborhood $V_\tau \subseteq \partial_e \mathrm{T}(A)$ of $\tau$ such that the set
\[
D_\tau = \{ a \in A \otimes \mathcal{K} \ | \ r \leq \ d_\gamma(a) \leq s, \ \forall \gamma \in V_\tau \}
\]
is nonempty and contractible.  It is then not unreasonable to hope that a selection theory argument over $\partial_e \mathrm{T}(A)$ might allow one to assemble a positive element with a prescribed rank function using members of the various $D_\tau$.  We prove here that the sets $D_\tau$ are contractible at least when $A$ is $\mathcal{Z}$-stable (see Remark \ref{othersets}), suggesting that this approach to building rank functions merits investigation under the lesser assumption of strict comparison.  It also suggests---and this is supported by the second author's so far vain efforts to resolve the rank function problem in ASH algebras---that the difficulty in completing the proof of the Toms-Winter conjecture may have a homotopy theoretic aspect.

\vspace{2mm}

\noindent
{\bf Acknowledgements.} The second author is indebted to Wilhelm Winter for motivating conversations on this topic.


\section{Preliminaries}
We refer the reader to the surverys \cite{ara2009ktheoryoperatoralgebrasclassification} and \cite{GardellaPerera} for the basic theory of the Cuntz semigroup and dimension functions.  For the convenience of the reader we do at least recall here that for a unital C$^*$-algebra $A$ and normalized quasitrace $\tau:A \to \mathbb{R}$, the associated dimension function on $A$ is given by
\begin{equation*}
d_\tau(a) = \lim_{n \to \infty} \tau\left(a^{1/n}\right)
\end{equation*}
for positive $a \in A$. Again with $A$ unital we use $\qt(A)$ to denote the set of all normalized quasitraces on $A$.  For any C$^*$-algebra let $\mathcal{P}(A)$ denote the set of projections in $A$.\\

In this section, all tensor products will be assumed to be complete in the spatial norm, and all C$^*$-algebras will be stably finite. The first proposition gives several useful characterizations of compactness (hence also non-compactness) of Cuntz classes.
\begin{prop}
\label{compact}
    Let $A$ be a unital, simple $C^*$-algebra with stable rank one. For any positive $a\in A$, the following are equivalent:
    \begin{enumerate}
        \item The Cuntz class $\langle a\rangle$ is compact.
        \item There exists some $\varepsilon>0$ such that $a\sim (a-\varepsilon)_+$.
        \item There exist some $\varepsilon>0$ and $\tau\in \qt(A)$ such that $d_\tau((a-\varepsilon)_+)=d_\tau(a)$.
        \item $0$ is an isolated point in $\s(a)$ or $0\notin \s(a)$.
        \item There exists some $p\in \proj(A)$ such that $a\sim p$.
    \end{enumerate}
\end{prop}
\begin{proof}
    \begin{enumerate}[(i)]
        \item $(1)\iff (2)$: This follows from \cite[Proposition 4.3]{GardellaPerera}.
        \item $(2)\implies (3)$: Stable finiteness ensures that $A$ admits a quasitrace $\tau$. Dimension functions are invariant under Cuntz equivalence, so we have $d_\tau((a-\varepsilon)_+)=d_\tau(a)$.
        \item $(3)\implies (4)$: Suppose there is a sequence $(e_n)\in \s (a)$ that converges to zero. Consider any $\varepsilon>0$ and $\tau\in \qt(A)$. There must exist some $N$ such that $e_N<\varepsilon$. Dimension functions preserve order, and quasitraces on simple algebras are faithful; so, \cite[Lemma 5.8]{ara2009ktheoryoperatoralgebrasclassification} gives
        \[d_\tau((a-\varepsilon)_+)\le d_\tau((a-e_N)_+)<d_\tau(a)\]
        \item $(4)\implies (5)$: This follows from \cite[Proposition 3.12]{FP_compact}.
        \item $(5)\implies (2)$: Suppose $p\sim a$. In particular, we have $p\precsim a$, so R{\o}rdam's lemma as stated in \cite[Proposition 2.4]{RORDAM1992255} shows that there exists some $\varepsilon>0$ such that $(a-\varepsilon)_+\succsim (p-1/2)_+=p/2\sim p\sim a$. Since we already have $(a-\varepsilon)_+\precsim a$, we get $\langle a\rangle =\langle (a-\varepsilon)_+\rangle$. \qedhere
    \end{enumerate}
\end{proof}

The next two results, the first simply a restatement of \cite[Proposition 5.9]{ara2009ktheoryoperatoralgebrasclassification}, will be used to confirm that the homotopies in the proof of the main theorem maintain Cuntz equivalence and non-compactness, respectively.
\begin{prop}
\label{Cuntz_equiv}
    Consider a simple $C^*$-algebra $A$ with strict comparison of positive elements and any $a,b\in A_+$ with non-compact Cuntz classes. Then $\langle a\rangle=\langle b\rangle$ if and only if $d_\tau(a)=d_\tau(b)$ for all $\tau\in \qt(A)$.
\end{prop}

\begin{lem}
\label{noncmpt_tensor}
Suppose $A$ and $B$ are unital, simple and $\js$-stable $C^*$-algebras. If $\langle a\rangle\in \Cu(A)$ is non-compact, then so is $\langle a\otimes b\rangle\in \Cu(A\otimes B)$ for any nonzero $b\in B_+$.
\end{lem}

\begin{proof}
    Proposition \ref{compact} shows that $\s(a)$ contains a positive sequence $x_n$ that converges to $0$. Since $b$ is nonzero and positive, its spectrum contains a positive element, say $y$. We know that $\s(a\otimes b)$ is the set of products $\s(a)\s(b)$ from \cite{tensor_spectrum}. So, $\s(a\otimes b)$ contains the positive sequence $x_n\cdot y$ converging to $0\cdot y=0$. Therefore, $\langle a\otimes b\rangle$ is not compact by Proposition \ref{compact}.
\end{proof}

The following lemma provides sufficient conditions for an element to be the supremum of a sequence of Cuntz classes. The proof below is taken from \cite[Theorem 3.8]{GardellaPerera}, first shown in \cite{CEI} using Hilbert C$^*$-modules.
\begin{lem}
\label{Cuntz_sup}
    Given positive elements $a$ and $a_n$ in a $C^*$-algebra $A$ such that $\displaystyle\lim_{n\to\infty}a_n=a$ and $a_1\precsim a_2\precsim \cdots\precsim a$, we have $\langle a\rangle =\sup_n\langle a_n\rangle$.
\end{lem}
\begin{proof}
    The class $\langle a\rangle$ is an upper bound for the sequence $\langle a_n\rangle$. It suffices to show that for any other upper bound $\langle b\rangle$ of $\langle a_n\rangle$, we have $\langle a\rangle\le \langle b\rangle$. By R{\o}rdam's lemma, we need only prove that $(a-\varepsilon)_+\precsim b$ for any $\varepsilon>0$. The convergence gives a natural number $N$ such that $\|a-a_{N}\|<\varepsilon$. By \cite[Lemma 2.2]{Kirchberg2002InfiniteNC}, there exists some $d\in A$ such that $(a-\varepsilon)_+=da_{N}d^*$, which, in turn, implies that $(a-\varepsilon)_+\precsim a_{N}\precsim b$.
\end{proof}
The corollary below will help deal with Cuntz classes in the stabilization.

\begin{cor}
\label{stable_sup}
For unital $C^*$-algebras $A$ and $B$ with $a\in (A\otimes \K)_+$, there exists a sequence $(a_n)\in M_n(A)$ satisfying the following:
\begin{enumerate}[(i)]
    \item The sequence $\langle a_n\rangle$ is increasing in Cu$(A)$ with supremum $\langle a\rangle$.
    \item For any $b\in B_+$, the sequence $\langle a_n\otimes b\rangle$ is increasing in Cu$(A\otimes B)$ with supremum $\langle a\otimes b\rangle$.
\end{enumerate}

\end{cor}
\begin{proof}
    Consider the usual inductive limit construction of the stabilization:
    \[A\hookrightarrow M_2(A) \hookrightarrow M_3(A)\hookrightarrow \cdots \hookrightarrow A\otimes \K\]\
    Let $1_n$ denote the unit in $M_n(A)$ and set $a_n=1_na1_n$. We will now show that the hypotheses of Lemma \ref{Cuntz_sup} are satisfied in both cases.
\begin{enumerate}[(i)]
    \item It is easy to see that $\displaystyle\lim_{n\to\infty}a_n=a$. That $a_1\precsim a_2\precsim \cdots \precsim a$ holds follows by observing that
    \[1_na_{n+1}1_n=1_n1_{n+1}a1_{n+1}1_n=1_na1_n=a_n\]
    \item Continuity of factors in a simple tensor gives $\displaystyle\lim_{n\to\infty}a_n\otimes b=a\otimes b$. That $a_1\otimes b\precsim a_2\otimes b\precsim \cdots \precsim a\otimes b$ holds is clear from the equation
    \[(1_n\otimes 1_B)(a_{n+1}\otimes b)(1_n\otimes 1_B)=a_n\otimes b=(1_n\otimes 1_B)(a\otimes b)(1_n\otimes 1_B)\qedhere\]
\end{enumerate}
\end{proof}

The next lemma produces an element of the Jiang-Su algebra $\js$ whose spectral properties will be instrumental in the proof of the main theorem.

\begin{lem}
\label{leb_elem_js}
    There exists a positive element $z_1\in \js$ with spectrum $[0,1]$ on which the unique normalized trace $\tau_\js$ on $\js$ induces the Lebesgue measure.
\end{lem}
\begin{proof}
    Let $\phi$ denote the faithful tracial state on $C[0,1]$ given by integrating against the Lebesgue measure $\lambda$ on $[0,1]$. Theorem 2.1 of \cite{rordam2004stablerealrankzabsorbing}, itself fashioned from \cite{xinhui_jiang__1998}, provides a unital embedding $\psi: C[0,1]\hookrightarrow\js$ such that $\tau_\js\circ \psi=\phi$. Fix such a map $\psi$. Let $i_1\in (C[0,1])_+$ denote the inclusion map $[0,1]\hookrightarrow \C$ and set $z_1=\psi(i_1)$. As $\psi$ is a $*$-homomorphism, we have that $z_1$ is positive too, and injectivity gives $\s(z_1)=\s(i_1)=[0,1]$. The trace $\tau_\js$ induces the Lebesgue measure on $\s(z_1)$, since, for any $f\in C(\s (z_1))$, we see that
    \[\tau_\js(f(z_1))=\tau_\js(f(\psi(i_1)))=\tau_\js(\psi(f(i_1))) =\phi(f)=\int_{[0,1]} f\,d\lambda\qedhere\]
\end{proof}
Recall that a quasitrace on a C$^*$-algebra $A$ is called a 2-quasitrace if it extends to one on $M_2(A)$ or, equivalently, to one on $A\otimes \K$. Denote by $\qt_2(A)$ the set of 2-quasitraces on $A$. For a complex-valued function $f$, we use $\coz(f)$ to denote the cozero set of $f$, i.e., the complement of the zero set of $f$.\\

The last lemma relates the value of dimension functions on simple tensors with those on the individual elements.
\begin{lem}
\label{dim_tensor}
Let $A$ and $B$ be unital $C^*$-algebras with $b\in B_+$.
\begin{enumerate}[(i)]
    \item For any positive $a\in A$ and $\tau\in \qt(A\otimes B)$, we have
    \[d_\tau(a\otimes b)=d_{\tau_A}(a)d_{\tau_B}(b)\]
    where $\tau_A$ and $\tau_B$ are the restrictions of $\tau$ on $A\otimes 1$ and $1\otimes B$ respectively.
    \item For any positive $a\in A\otimes \K$ and $\tau\in \qt(A\otimes\K\otimes B)=\qt_2(A\otimes B)$, we have
    \[d_\tau(a\otimes z)=d_{\rho}(a)d_{\tau_\js}(z)\]
    where $\rho$ and $\tau_B$ are the induced quasitraces on $A\otimes \K$ and $B$ respectively.
\end{enumerate}

\end{lem}

\begin{proof}
\begin{enumerate}[(i)]
    \item The quasitraces $\tau_A$ and $\tau_B$ restrict to traces on the commutative C$^*$-subalgebras $C^*(1\otimes 1,a\otimes 1)\cong C^*(1,a)\cong C(\s (a))$ and $C(\s (b))$ respectively. Setting $D=C^*(1\otimes 1,a\otimes 1, 1\otimes b)\subset A\otimes B$, we get the chain of $*$-isomorphisms
    \[D\cong C^*(1,a)\otimes C^*(1,b)\cong C(\s (a))\otimes C(\s (b))\cong C(\s (a)\times\s (b))\]
    For any $v\in D$, let $g_v$ denote its image in $C(\s (a)\times\s (b))$ under the isomorphisms above. Also, let $g_w$ denote the inclusion map $\s (w)\hookrightarrow \C$ for any $w\in A$ or $w\in B$. The map $g_{a\otimes b}$ is given by
    \begin{equation}
    \label{eqn:4}
        g_{a\otimes b}(s,t)=g_{a\otimes 1}(s,t)g_{1\otimes b}(s,t)=g_a(s)g_b(t)=st
    \end{equation}
    We, therefore, have $\coz (g_{a\otimes b})=\coz (g_a)\times \coz (g_b)$. Let $\mu$, $\mu_a$, and $\mu_b$ denote the measures induced by $\tau$, $\tau_A$, and $\tau_B$ on $C(\s (a)\times\s (b))$, $C(\s (a))$, and $C(\s (b))$ respectively. \cite[Corollary 8.6]{Godfrey_Sion_1969} shows that $\mu$ is the product measure $\mu_a\times \mu_b$. Dimension functions are induced measures of cozero sets in the functional calculus by \cite[Proposition I.2.1]{BLACKADAR1982297}, so
    \[d_\tau(a\otimes b)=\mu(\coz (g_{a\otimes b}))=\mu_a(\coz(g_a))\mu_b(\coz (g_b))=d_{\tau_A}(a)d_{\tau_B}(b)\]
    \item Let $\tau_n$ (resp. $\rho_n$) denote the quasitrace induced on $M_n(A\otimes B)$ by $\tau$ (resp. on $M_n(A)$ by $\rho$). Since dimension functions and preserve suprema, we use the sequence $a_n$ obtained in Corollary \ref{stable_sup} to get
    \begin{align*}
        d_\tau(a\otimes b)&=d_\tau(\langle a\otimes b\rangle)\\
        &=\sup_nd_\tau(\langle a_n\otimes b\rangle)\\
        &=\sup_nd_{\tau_n}(a_n\otimes b)\\
        &\stackrel{(i)}{=} \sup_nd_{\rho_n}(a_n)d_{\tau_B}(b)\\
        &=\sup_nd_{\rho}(a_n)d_{\tau_B}(b)\\
        &=d_{\rho}(a)d_{\tau_B}(b)\qedhere
    \end{align*}
\end{enumerate}
\end{proof}

The proposition below helps translate the results of Zhang, Jiang and Hua into the first part of the statement of Corollary \ref{homgroups}.
\begin{prop}
\label{proj_def_ret}
    Consider a unital, simple $C^*$-algebra $A$ with stable rank one. If $a\in A_+$ has compact Cuntz class, then $S_a$ deformation retracts onto the set $\{p\in \proj(A): p\sim a\}$.
\end{prop}

\begin{proof}
    Consider any element $v\in S_a$. By Proposition \ref{compact}, there exists some $\delta>0$ such that $\s(v)$ contains no positive elements less than $\delta$. Therefore, the characteristic function of the set $\s(v)\setminus\{0\}$, denoted as $f_1$, is contained in $C(\s(v))$. For $t\in [0,1)$ and $s\in \s(v)$, define a piecewise-linear function
    \[f_t(s)=\begin{cases}
    s/(1-t), & 0\le s<1-t \\
    1, & 1-t\le s<1\\
    1-(1-t)(1-s), & s\ge 1
   \end{cases}\]
   That $f$ is continuous in $s$ is clear. We show continuity in $t$. Consider the case when $\|v\|>1$. For any $r, t\in [0,1]$ and $\varepsilon>0$ such that $0\le t-r<\varepsilon$, the function $|f_t-f_r|$ is increasing on $(0,1-t)\cup (1,\|v\|)$, decreasing on $(1-t,1-r)$, and zero on $[1-r,1]$. If $r=1$, then that forces $t=1$ and $\|f_t-f_r\|_\infty=0$. We can now assume that $r<1$.
   \begin{align*}
       \|f_t-f_r\|_\infty&\le\max\{(f_t-f_r)(1-t),(f_r-f_t)(\|v\|)\}\\
       &=\max\left\{1-(1-t)/(1-r),(1-r)(\|v\|-1)-(1-t)(\|v\|-1)\right\}\\
       &=\max\left\{(t-r)/(1-r),(t-r)(\|v\|-1)\right\}\\
       &<\varepsilon\cdot\max\left\{1/(1-r),(\|v\|-1)\right\}
   \end{align*}
   When $\|v\|\le 1$, then the inequality will just be $\|f_t-f_r\|_\infty\le \varepsilon/(1-r)$. In both cases, we get $\|f_t-f_r\|_\infty\to 0$ as $\varepsilon\to 0$, proving the continuity of $f$ in $t$.\\
   
   Consider the map $F:[0,1]\times S_a\to A$ given by $F_t(v)=f_t(v)$. Continuity of $F$ follows from that of $f$ and of the functional calculus. Since $f_t$ has the same support as $f_0$, the identity map on $\s(v)$, for all $t\in [0,1]$, we see that $F_t(v)\sim v\sim a$ by \cite[Proposition 2.5]{ara2009ktheoryoperatoralgebrasclassification} and $F$ is a deformation retraction. The map $f_1$ is a projection in $C(\s(v))$, hence so is $F_1(v)$. As $F_t$ fixes projections for all $t\in [0,1]$, the image of $F_1$ must be the set $\{p\in \proj(A): p\sim a\}$. 
\end{proof}

\section{Proof of the Main Theorem}\label{proof}

We are now ready to prove the main result restated here:

\begin{customthm}{I}
    Consider a unital, simple, and $\js$-stable $C^*$-algebra $A$. For any $a\in A_+$ with non-compact Cuntz class, the set $S_a=\{b\in A_+: a\sim b\}$ is a contractible subspace of $A$. Moreover, for any $a\in (A\otimes\K)_+$ with non-compact Cuntz class, the class $\langle a\rangle$ is a contractible subspace of $A\otimes \K$.
\end{customthm}

\begin{proof}
If $A$ is purely infinite, any two nonzero positive elements are Cuntz equivalent. So, for any $a\in A_+\setminus \{0\}$, we have $S_a = A_+\setminus \{0\}$. The straight-line homotopy $\varphi_t(b)=ta+(1-t)b$ witnesses the contractibility of $S_a$. By the dichotomy in \cite[Theorem 3]{Gong_Jiang_Su_2000}, we can now assume that $A$ is stably finite.\\

Fix some positive $a\in A$ and $*$-isomorphisms $\alpha : A\to A\otimes\js\otimes\js$ and $\theta: \js\otimes\js \to \js$. We identify $S_a$ with its images under the $*$-isomorphisms $\theta$ and $(\id_A\otimes\theta)\circ\alpha: A\to A\otimes\js$. The proof will construct a concatenated deformation retraction from $S_a\subset A\otimes \js\otimes\js$ onto a singleton of the form
\begin{equation}
\label{def_ret}
\varphi(t)= \begin{cases}
    \varphi_1(3t), & 0\le t\le1/3\\
    \varphi_2(3t-1), & 1/3<t<2/3 \\
    \varphi_3(3t-2), & 2/3\le t\le 1
   \end{cases}
\end{equation}
\subsection{The First Homotopy}
\label{first_hom}
The unital $*$-homomorphisms $\theta^{-1}$ and $\id_\js \otimes 1_\js$ are strongly asymptotically unitarily equivalent by \cite[Theorem 2.2]{DW}. In other words, for all $t\in [0,1)$, there exists a unitary $u_t\in \js\otimes\js$ such that $u_0=1_{\js\otimes\js}$ and the map $H:[0,1]\times \js\to \js\otimes\js$ given below is a homotopy:
\[ H_t= \begin{cases}
    \Ad(u_t)\circ\theta^{-1}, & 0\le t<1 \\
    \id_\js\otimes 1_\js, & t=1
   \end{cases}
\]
Thus, the homotopy below is a deformation retraction from $A\otimes\js\otimes\js$ onto $A\otimes\js\otimes 1_\js$:
\[G_t=\id_{A}\otimes (H_t\circ \theta)=\begin{cases}
    \id_A\otimes\Ad(u_t), & 0\le t<1 \\
    \id_A\otimes\theta\otimes 1_\js, & t=1
   \end{cases}
\]
Let $\varphi_1(t)$ be the restriction of $G_t$ on $S_a\subset A\otimes\js\otimes\js$. We show that $\varphi_1(t)$ is a self-map on $S_a$ for all $t\in[0,1]$: for any $b\in S_a\subset A\otimes \js\otimes\js$, we have
\begin{align}
    \lim_{n\to\infty}(1_A\otimes u_{t(1-1/n)})b(1_A\otimes u_{t(1-1/n)})^*&= \varphi_1(t)(b)\label{eqn:2}\\
    \lim_{n\to\infty}(1_A\otimes u_{t(1-1/n)})^*(\varphi_1(t)(b))(1_A\otimes u_{t(1-1/n)})&= b\label{eqn:3}
\end{align}
So, the homotopy $\varphi_1$ is a deformation retraction from $S_a\subset A\otimes\js\otimes\js$ onto $S_a\otimes1_\js\subset A\otimes\js\otimes1_\js$.
\subsection{The Second Homotopy}
The element $z_1\in \js$ from Lemma \ref{leb_elem_js} is clearly homotopic in $\js_+$ to $1_\js$ via the straight line $z_t=t\cdot z_1+(1-t)1_\js$. For any $b\in S_a\subset A\otimes\js\otimes\js$, we set $c_b=(\id_A\otimes \theta)(b)\in S_a\subset A\otimes \js$ and
\[\varphi_2(t)(b)=c_b\otimes z_{t}\in S_a\otimes \js\subset A\otimes \js\otimes\js\]
The non-compactness of any Cuntz class of the form $\langle c_b\otimes z_{t} \rangle$ follows from Lemma \ref{noncmpt_tensor}. Functional calculus shows that for any $t\in[0,1]$, the cozero sets of $z_t$ and $1_\js$ in $C^*(1_\js,z_1)$ can differ only by the Lebesgue-null set $\{0\}$. So, we have $d_{\tau_\js}(z_t)=\lambda(\coz(z_t))=\lambda(\coz(1_\js)) =d_{\tau_\js}(1_\js)=1$. Since dimension functions are invariant under Cuntz equivalence, Lemma \ref{dim_tensor} gives us that for any $t\in[0,1]$, $b\in S_a\subset A\otimes\js\otimes\js$, and $\tau\in \qt(A\otimes\js\otimes\js)$, we have
\begin{align*}
    d_\tau(\varphi_2(t)(b))&=d_\tau(c_b\otimes z_t)\\
    &=d_{\tau_{A\otimes \js}}(c_b)d_{\tau_\js}(z_t)\\
    &=d_{\tau_{A\otimes \js}}((\id_A\otimes \theta)(b))d_{\tau_\js}(1_\js)\\
    &=d_\tau((\id_A\otimes \theta)(b)\otimes 1_\js)\\
    &=d_\tau(\varphi_1(1)(b))\\
    &=d_\tau(b)\nonumber\\
    &=d_\tau(\alpha(a))
\end{align*}
Therefore, Proposition \ref{Cuntz_equiv} shows that $\varphi_2$ is a self map of $S_a\subset A\otimes\js\otimes\js$ with image $S_a\otimes z_1\subset A\otimes \js\otimes z_1$.
\subsection{The Third Homotopy}
Consider the elements $r_t=(z_1-t)_+$ and $s_t=(t-z_1)_+$ in $\js$ for any $t\in [0,1]$. Functional calculus shows that $r_t$ and $s_t$ are orthogonal elements satisfying $d_\tau(r_t)=1-t$, $d_\tau(s_t)=t$, $r_0=z_1$, and $r_1=s_0=0$. Let $p=(\id_A\otimes \theta)(\alpha(a))$. For all $b\in S_a\subset A\otimes\js\otimes\js$ and $t\in [0,1]$, we set
\[\varphi_3(t)(b)=c_b\otimes r_t+p\otimes s_t\in S_a\otimes \js\subset A\otimes \js\otimes\js\]
By Lemma \ref{noncmpt_tensor}, both simple tensors on the right have non-compact Cuntz classes, so $0$ is an accumulation point in their spectra. They are orthogonal too, so the spectrum of the sum is the union of that of the summands, which must also have $0$ as an accumulation point. Therefore, the Cuntz class of any element in the image of $\varphi_3$ is not compact. Orthogonality also means that $\langle \varphi_3(t)(b)\rangle = \langle c_b\otimes r_t\rangle +\langle p\otimes s_t\rangle$ by \cite[Lemma 2.10]{ara2009ktheoryoperatoralgebrasclassification}. We use Lemma \ref{dim_tensor} and the fact that dimension functions are additive on Cuntz classes to observe that for any $t\in[0,1]$, $b\in S_a\subset A\otimes\js\otimes\js$, and $\tau\in \qt(A\otimes\js\otimes\js)$:
\begin{align*}
    d_\tau(\varphi_3(t)(b))&=d_\tau(\langle c_b\otimes r_t+p\otimes s_t\rangle)\\
    &=d_\tau(\langle c_b\otimes r_t\rangle +\langle p\otimes s_t\rangle)\\
    &=d_\tau(\langle c_b\otimes r_t\rangle)+d_\tau(\langle p\otimes s_t\rangle)\\
    &=d_\tau(c_b\otimes r_t)+d_\tau(p\otimes s_t)\\
    &=d_{\tau_{A\otimes \js}}(c_b)d_{\tau_\js}(r_t)+d_{\tau_{A\otimes \js}}(p)d_{\tau_\js}(s_t)\\
    &=d_{\tau_{A\otimes \js}}(p)(1-t)+d_{\tau_{A\otimes \js}}(p)t\\
    &=d_{\tau_{A\otimes \js}}(p)\cdot 1\\
    &=d_{\tau_{A\otimes \js}}((\id_A\otimes \theta)(\alpha(a)))d_{\tau_\js}(1_\js)\\
    &=d_\tau(\alpha(a))
\end{align*}
So, the homotopy $\varphi_3$ is a self-map of $S_a\subset A\otimes\js\otimes\js$ by Proposition \ref{Cuntz_equiv}. Concatenating $\varphi_i$ as in Equation (\ref{def_ret}) yields a deformation retraction $\varphi$ from $S_a\subset A\otimes\js\otimes\js$ onto the singleton $\varphi(1)(S_a)=\varphi_3(1)(S_a)=\{p\otimes s_1\}$ and completes the proof of the first statement of the theorem.\\

The same argument with $A\otimes \K$ instead of $A$ proves the second statement as well after replacing $1_A$ with the unit of $M_n(A)$ in Equations (\ref{eqn:2}) and (\ref{eqn:3}).
\end{proof}

\begin{remark}\label{othersets}
{\rm The methods developed here apply equally well to various other subsets of positive elements defined in terms of their Cuntz classes.  For instance one could consider, for an algebra $A$ as in Theorem \ref{main_result}, real numbers $0<s<r$, and a nonempty closed face $K$ of $\qt(A)$, the set
\[
D = \{ a \in A_+ \ | \ s \leq d_\tau(a) \leq r \ \forall\, \tau \in K \}.
\]
The arguments of Section \ref{proof} will show that, {\it mutatis mutandis}, $D$ is contractible.  }
\end{remark}

\bibliography{Cuntzcontractible}
\bibliographystyle{abbrv}
\end{document}